\documentclass[12pt]{article}
\usepackage{amsmath, amsthm}\usepackage{enumerate}\usepackage{amssymb}
\usepackage{bbold}
\usepackage{bbm}\usepackage{dsfont}\usepackage{color}\usepackage{xcolor}
\usepackage[explicit]{titlesec}
\newtheorem{theorem}{Theorem}[subsection]
\newtheorem{proposition}[theorem]{Proposition}

\newtheorem{lemma}[theorem]{Lemma}
\newtheorem{example}[theorem]{Example}
\newtheorem{corollary}[theorem]{Corollary}
\newtheorem{remark}[theorem]{Remark}

\newtheorem{definition}[theorem]{Definition}
\newcommand{\sol}{{\text{\rm sol}}}

\makeatletter
\renewcommand{\subsection}{\@startsection{subsection}{1}
{0pt}{3.25ex plus 1ex minus.2ex}{-1em}{\normalfont\normalsize\bf}}\makeatother
\begin{document}

%%%%Page1%%%
\title{{\bf On aL(M)wc and oL(M)wc operators between Banach lattices}}
\maketitle
\author{\centering{{Safak Alpay$^{1}$, Svetlana Gorokhova $^{2}$\\ 
\small $1$ Middle East Technical University, Ankara, Turkey\\ 
\small $2$ Uznyj matematiceskij institut VNC RAN, Vladikavkaz, Russia}
\abstract{We study almost L(M) weakly compact and order L(M) weakly compact operators in Banach spaces. 
Several further topics related to these operators are investigated.}

\vspace{5mm}

\vspace{5mm}
}}

\bigskip
\bigskip
\bigskip

%Various types of compactness in Banach spaces were in streamline 
%of functional analysis in the second half of the 20th century.  
%At some point, in order to diversify the classes of compact operators, 
%the Banach lattice structure had come in play by Meyer-Nieberg, 
%Dodds, Fremlin, Aliprantis, Wickstead, and others.
In the present paper, we continue to study compactness in Banach spaces with main 
emphasis on almost and order L(M)-weakly compact operators (see, also 
the recent papers \cite{AG_la},  \cite{OM22}, 
\cite{EG23aff}, \cite{AEG_duality}, and \cite{AEG_env}).

%%%%%%%%%%%%%%%%%%%%
\section{Introduction and Preliminaries}
%%%%%%%%%%%%%%%%%%%%

Throughout the paper, all vector spaces are real and operators are linear and bounded; $X$ and $Y$ 
denote Banach spaces; $E$, $F$, and $G$ denote Banach lattices. A subset $A$ of $X$ is called
{\em bounded} if $A$ is norm bounded. We denote by $B_X$ the closed unit ball of $X$,
by $\text{\rm L}(X,Y)$ ($\text{\rm W}(X,Y)$, $\text{\rm K}(X,Y)$) 
the space of all bounded (resp., weakly compact, compact) operators 
from $X$ to $Y$, by $I_X$ the identity operator on $X$. 
We also denote by $E_+$ the positive cone of $E$, by $\text{sol}(A) = \bigcup_{a\in A}[-|a|,|a|]$ the solid hull 
of $A\subseteq E$, by $$E^a=\{x\in E: |x|\ge x_n\downarrow 0\Rightarrow\|x_n\|\to 0\}$$ 
the $\text{\rm o}$-continuous part of $E$, and by $\text{\rm L}_r(E,F)$ (resp., $\text{\rm L}_{ob}(E,F)$) the space of regular
(resp., order bounded) operators from $E$ to $F$.

\subsection{}
We shall use the following definition.

\begin{definition}\label{limited subset}%no
{\em
A bounded subset $A$
\begin{enumerate}[a)]
\item 
of $X$ is called {\em limited}, 
if each \text{\rm w}$^\ast$-null  sequence in $X^\ast$ is uniformly null on $A$.
\item
of $X$ is called {\em Dunford-Pettis} or a \text{\rm DP}-{\em set}
if each weakly null  sequence in $X^\ast$ is uniformly null on $A$.
\item 
of $E$ is called {\em almost limited} if each disjoint 
\text{\rm w}$^\ast$-null sequence in $E^\ast$ is uniformly null on $A$.
\item 
of $E$ is called an \text{\rm Lwc} {\em set} if each disjoint 
sequence in $\sol(A)$ is norm-null.
\item 
of an ordered Banach space $X$ is called {\em almost order bounded} if, 
for each $\varepsilon>0$ there exist $a,b\in X$ satisfying 
$A\subseteq[a,b]+\varepsilon B_X$.
\end{enumerate}}
\end{definition}

\noindent
For every \text{\rm Lwc} subset $A$ of $E$, 
we have $A\subseteq E^a$. Indeed, otherwise, for some $a\in A$ with $|a|\in E\setminus E^a$ 
there exists a disjoint sequence 
$(x_n)$ in $[0,|a|]\subseteq\text{\rm sol}(A)$
with $\|x_n\|\not\to 0$.

\subsection{}
The next definition will be used often.

\begin{definition}\label{Main Schur property}%no
{\em A Banach space $X$ has:
\begin{enumerate}[a)]
\item 
the {\em Schur property} (or $X\in\text{\rm (SP)}$)
if each weakly null sequence in $X$ is norm null; % (cf. \cite[p.\,207]{AlBu});
\item
the  w$^\ast$-{\em Dunford--Pettis property} (or $X\in\text{\rm (w$^\ast$DPP)}$) if 
%the  {\em Dunford--Pettis$^\ast$ property} (or $X\in\text{\rm (DP$^\$P)}$) if 
%weakly compact subsets of $X$ are limited. 
$f_n(x_n)\to 0$ for each weakly null $(x_n)$ in $X$ and each weak$^\ast$ null $(f_n)$ in $X^\ast$.
\end{enumerate}}
\end{definition}

\begin{definition}{\em
A Banach lattice $F$ has the {\em property} (d) (shortly $F\in\text{\rm (d)}$) if, 
for every disjoint $\text{\rm w}^\ast$-null $(f_n)$ in $F^\ast$,
$|f_n|\stackrel{\text{\rm w}^\ast}{\to}0$ \cite[Def.\,1]{El}.
}
\end{definition}

\noindent
Every $\sigma$-Dedekind complete Riesz space has (d) \cite{El}. 
The Banach lattice $c$ does not have (d) \cite[Example~2.1\,(2)]{CCJ}.

\subsection{}
We use also the following definitions concerning operators.

\begin{definition}\label{X to Y}%no
{\em An operator $T: X\to Y$ is called 
{\em limited} ($T\in\text{\rm Lim}(X,Y)$) if $T(B_X)$ is a limited subset of $Y$. }
\end{definition}

\noindent
The canonical embedding of $c_0$ into $\ell^\infty$ is limited, but not compact.

\begin{definition}\label{X to F}%OK
{\em An operator $T:X\to F$ is called
\begin{enumerate}[a)]
\item
{\em almost limited} ($T\in\text{\rm aLim}(X,F)$) if $T(B_X)$ is almost limited. 
\item 
{\em semi-compact} ($T\in\text{\rm semi-K}(X,F)$) if it carries bounded subsets of $X$ onto almost order bounded subsets of $F$.
\item  
$\text{\rm L}$-{\em weakly compact} ($T\in\text{\rm Lwc}(X,F)$) 
if $T$ carries bounded subsets of $X$ onto \text{\rm Lwc} subsets of $F$;
\item
{\em almost} $\text{\rm L}$-{\em weakly compact} 
($T\in\text{\rm aLwc}(X,F)$) 
if $T$ carries relatively weakly compact subsets of $X$ onto 
$\text{\rm Lwc}$ subsets of $F$ \cite[Def.\,2.1]{BLM18}.
\item
{\em limitedly {\rm L}-weakly compact} ($T\in l\text{\rm -Lwc}(X,F)$)
if $T$ maps limited subsets of $X$ onto  \text{\rm Lwc} subsets of $F$ (see \cite[Def.\,2.8\,(b)]{AEG_duality}).
\end{enumerate}}
\end{definition}

\noindent
It is well known that $\text{\rm semi-K}(X,F)=\text{\rm Lwc}(X,F)$ for every $X$,
whenever the norm in $F$ is \text{\rm o}-continuous.
It is easily seen that the identity operator $I_E$
is an $l\text{\rm -Mwc}$ operator iff $E^\ast$ is a KB-space.
An \text{\rm Lwc} {\em operator} $T:X\to F$ is defined in \cite[Def.\,1.\,iii)]{Mey74}
via the condition that $T(B_X)$ is an \text{\rm Lwc} set. Then 
$\text{\rm Lwc}(X,F)\subseteq\text{\rm W}(X,F)$, and
\begin{equation}\label{1a}%OK
   T\in\text{\rm Lwc}(X,F)\Longleftrightarrow\text{\rm $T$ 
takes bounded subsets onto \text{\rm Lwc} subsets.}
\end{equation}

\begin{definition}\label{E to Y}%OK
{\em An operator $T:E\to Y$ is called
\begin{enumerate}[a)]
\item
%{\em order weakly compact} ($T\in\text{\rm owc}(E,Y)$)
%if $T$ carries order intervals of $E$ onto relatively 
%weakly compact subsets of $Y$~\cite[Def.\,2.1]{Dodds}.
%\item 
M-{\em weakly compact} ($T\in\text{\rm Mwc}(E,Y)$) if $\|Tx_n\|\to 0$ 
for every disjoint bounded $(x_n)$ in $E$ \cite[Def.\,1(iv)]{Mey74}$;$
\item
{\em limitedly {\rm M}-weakly compact} ($T\in l\text{\rm -Mwc}(E,Y)$), 
if $Tx_n\stackrel{\text{\rm w}}{\to}0$ for every disjoint bounded $(x_n)$ in $E$. 
\item 
{\em almost \text{\rm M}-weakly compact} 
($T\in\text{\rm aMwc}(E,Y)$)  
if $f_n(Tx_n)\to 0$ for every weakly convergent 
$(f_n)$ in $Y^\ast$ and every disjoint 
bounded $(x_n)$ in $E$ \cite[Def.\,2.2]{BLM18}.
\end{enumerate}}
\end{definition}

\noindent
The classes of $l$-L(M)wc operators were defined in   \cite{AEG_duality} and in \cite{OM22} 
(in the second paper they were named  weak \text{\rm L(M)wc} operators).
Every Lwc (resp., Mwc) operator is aLwc (resp., aMwc) and $l$-Lwc (resp., $l$-Mwc).  However, the reverse inclusions are not true. 

\begin{example} 
$I_{\ell^1}\in\text{\rm aLwc}(\ell^1)\cap l\text{\rm -Lwc}(\ell^1)\setminus\text{\rm Lwc}(\ell^1)$;\\
$I_{c_0}\in\text{\rm aMwc}(c_0)\setminus\text{\rm Mwc}(c_0)$; \  \ 
$I_{\ell^\infty}\in l\text{\rm -Mwc}(\ell^\infty)\setminus\text{\rm Mwc}(\ell^\infty)$.
\end{example}

\begin{definition}\label{Main o-W operators}%OK
{\em An operator $T:E\to F$ is called
\begin{enumerate}[a)]
\item
{\em order} \text{\rm L}-{\em weakly compact} 
($T\in\text{\rm oLwc}(E,F)$) 
if $T$ carries order bounded subsets of $E$ onto 
$\text{\rm Lwc}$ subsets of $F$ \cite[Def.2.1]{BLM21}$;$
\item 
{\em order} \text{\rm M}-{\em weakly compact} 
($T\in\text{\rm oMwc}(E,F)$) 
if $f_n(Tx_n)\to 0$ for every order bounded $(f_n)$ in $F^\ast$ and every 
disjoint bounded $(x_n)$ in $E$~\cite[Def.2.2]{BLM21}.
\end{enumerate}}
\end{definition}

\begin{example}{\em
$I_{c_0}\in\text{\rm oLwc}(c_0)\setminus\text{\rm Lwc}(c_0)$; 
$I_{\ell^\infty}\in \text{\rm oMwc}(\ell^\infty)\setminus\text{\rm Mwc}(\ell^\infty)$.
}
\end{example}

\begin{definition}\label{DPOcoll}%no
{\em An operator $T: X\to Y$ is called 
\begin{enumerate}[]
\item[a)]
{\em Dunford--Pettis} (shortly, $T\in\text{\rm DP}(X,Y)$) 
if $T$ takes weakly null sequences to norm null ones;
\item[b)]
{\em weak Dunford--Pettis} (shortly, $T\in\text{\rm wDP}(X,Y)$)
if, for every w-null $(x_n)$ in $X$ and
w-null $(f_n)$ in $Y^\ast$ we have $f_n(Tx_n)\to 0$.
\end{enumerate}
An operator $T:E\to Y$ is called
\begin{enumerate}[]
\item[c)]
{\em almost Dunford--Pettis} ($T\in\text{\rm aDP}(E,Y)$) 
if $T$ takes disjoint w-null sequences to norm null ones \cite{Wnuk94}.
\item[d)]
{\em almost weak Dun\-ford--Pet\-tis} ($T\in\text{\rm awDP}(E,Y)$) 
if $f_n(Tx_n) \to 0$ whenever $(f_n)$ is w-null in $Y^\ast $ and 
$(x_n)$ is disjoint \text{\rm w}-null in $E$. 
\end{enumerate}}
\end{definition}

%In general, a \text{\rm DP}-operator need not be limited, and a limited operator need not be DP. 

\begin{example}{\em
The following holds.
\begin{enumerate}[(1)]
\item
$I_{\ell^1} \in\text{\rm DP}(\ell^1)\setminus \text{\rm Lim}(\ell^1)$.
\item
$J\in\text{\rm Lim}(c_0,\ell^\infty)\setminus\text{\rm DP}(c_0,\ell^\infty)$, where
$J:c_0\to \ell^\infty$ is the natural embedding. 
\item
$I_{c_0} \in\text{\rm wDP}(c_0)$ as $c_0\in\text{\rm (DPP)}$;
$I_{c_0} \not\in\text{\rm DP}(c_0)\cup\text{\rm aDP}(c_0)$. $I_{c_0} \in\text{\rm awDP}(c_0)$.
However, if $Y$ is reflexive, then $\text{\rm DP}(E,Y)=\text{\rm wDP}(E,Y)$.
\item
$I_{L^1[0,1]} \in\text{\rm aDP}(L^1[0,1])\setminus \text{\rm DP}(L^1[0,1])$.
\item
$I_{L_p[0,1]} \in\text{\rm aw$^\ast$DP}(L_p[0,1])$ for $1\leq p <\infty$,
\end{enumerate}
}
\end{example}

\noindent
For further unexplained terminology and notation, we refer to \cite{AlBu,BD,AAT,AEG22,BLM21,BLM18,EAS,Mey91}.

\medskip
\noindent
We start by investigating the algebraic properties of $\text{\rm aLwc}/\text{\rm aMwc}$ 
operators and show when they are closed ideals of $\text{\rm L}(E)$. 
Then we show $\text{\rm aL(M)wc}$ operators may coincide with bounded ones. 
We continue  with conditions for $\text{\rm aL(M)wc}$ operators to coincide 
with $\text{\rm L(M)wc}$ operators. Then we compare $\text{\rm aL(M)wc}$ operators 
with almost limited  and various Dunford -- Pettise type operators. 
In Section 3, we study similar questions for order $\text{\rm L(M)wc}$ 
and limited $\text{\rm L(M)wc}$ operators, and show each $T\in\text{\rm L}_r(E,F)$
is locally $\text{\rm oMwc}$. 
In Section 4, we discuss various class of operators that are considered. 

%%%%%%%%%%%%%%%%%%%%
\section{\text{\rm aL(M)wc} operators}
%%%%%%%%%%%%%%%%%%%%
The classes of almost  L-weakly compact (\text{\rm aLwc}) and  almost  
M-weakly compact  (\text{\rm aMwc}) operators were introduced in \cite{BLM18}. 
They generalize L- and M-weakly compact operators. 
Note that {\text{\rm aL(M)wc} operators were also studied in \cite{EAS,NoLE22,BLM21}.

\subsection{} 
We recall several facts on \text{\rm aLwc} operators.
As weakly compact sets are bounded, 
each \text{\rm Lwc} operator is an \text{\rm aLwc} operator.
\noindent
By \eqref{1a}, if $S\in\text{\rm W}(Y,X)$ and $T\in\text{\rm aLwc}(X,F)$ then
$T  S\in\text{\rm Lwc}(Y,F)$.
Furthermore, each \text{\rm aLwc} operator $T:X\to F$ is bounded. Indeed, otherwise
there exists a norm null sequence $(x_n)$ in $B_X$ satisfying $\|Tx_n\|\ge n$
for all $n\in\mathbb{N}$, violating the fact that the \text{\rm Lwc} set 
$\{Tx_n: n\in\mathbb{N}\}$ is bounded. Since $\text{\rm Lwc}(X,F)\subseteq\text{\rm W}(X,F)$,
the identity operator $I_{\ell^1}$ is not \text{\rm Lwc}. 

\subsection{Algebraic properties of \text{\rm aL(M)wc} operators.}
%%%%%%%%%%%%addition%%%%%%%%%%%

\begin{lemma}\label{Composition of Mwc}%OK
The following holds. 
\begin{enumerate}[\em (i)]
\item if $T\in\text{\rm aMwc}(E,F)$ and $S\in\text{\rm Mwc}(F,Y)$ then $S  T\in\text{\rm Mwc}(E,Y)$$;$
\item if $T\in\text{\rm aMwc}(E,X)$ and $S\in\text{\rm L}(X,Y)$ then $S  T\in\text{\rm aMwc}(E,Y)$$;$
\item if $T\in \text{\rm aLwc}(F,G)$, $S\in \text{\rm semi-K}(E,F)$, and $F=F^a$ then \\ $TS\in \text{\rm Lwc}(E,G)$.
\end{enumerate}
\end{lemma}

\begin{proof}
(i) Let $T\in\text{\rm aMwc}(E,F)$ and $S\in\text{\rm Mwc}(F,Y)$. 
Then $T^\ast \in\text{\rm aLwc}(F^\ast ,E^\ast )$ by \cite[Thm.2.5]{BLM18},
and $S^\ast \in\text{\rm Lwc}(Y^\ast ,F^\ast )$ by \cite[Prop.\,3.6.11\,(i)]{Mey91},
and hence $S^\ast$ is weakly compact. 
Then $(S  T)^\ast =T^\ast   S^\ast \in\text{\rm Lwc}(Y^\ast ,E^\ast )$
by \cite[Thm.2.4]{BLM18}. 
Thus $S  T\in\text{\rm Mwc}(E,Y)$  by \cite[Prop.\,3.6.11\,(i)]{Mey91}.

\medskip
(ii) Follows directly from definition of \text{\rm aMwc} operators. 

\medskip
(iii)
Since $S\in \text{\rm semi-K}(E,F)$, for each $\varepsilon >0$ 
there is $x\in E_+$ such that $S(B_E)\subseteq [-x,x]+\varepsilon B_F$.
As $F=F^a$, the set  $[-x,x]+\varepsilon B_F$ is relatively weakly compact in $F$. 
Since $T\in\text{\rm aLwc}(F,G)$, the set $T([-x,x]+\varepsilon B_F)=TS(B_E)$ 
is $\text{\rm Lwc}$. Hence $TS\in \text{\rm Lwc}(E,G)$.
\end{proof}

\begin{proposition}\label{Lwc-algebra}%OK (1)
For any Banach lattice $E$ the following hold.
\begin{enumerate}[{\rm (i)}]
\item $\text{\rm Lwc}(E)$ is a closed subalgebra of $\text{\rm W}(E)$ and is unital 
iff the identity operator $I_E$ is \text{\rm Lwc}.
\item $\text{\rm aLwc}(E)$ is a closed right ideal in $\text{\rm L}(E)$ 
$($and hence a subalgebra of $\text{\rm L}(E)$$)$, and is unital iff $I_E$ is {\em aLwc}.
\end{enumerate}
\end{proposition}

\begin{proof}
It is easy to see that both $\text{\rm Lwc}(E)$ and $\text{\rm aLwc}(E)$ are 
closed subspaces of $\text{\rm L}(E)$, and $\text{\rm Lwc(E)}\subseteq\text{\rm W}(E)$. 
The closedness of $\text{\rm Lwc}(E)$ under composition and 
the fact that $\text{\rm aLwc}(E)$ is a right ideal in $\text{\rm L}(E)$ follow
from \cite[Thm.2.4]{BLM18}.
The conditions on $I_E$, those make the algebras $\text{\rm Lwc}(E)$ and $\text{\rm aLwc}(E)$ unital,
are simply verified.
\end{proof}

\begin{proposition}\label{Mwc-algebra}%no
For any Banach lattice $E$ the following holds.
\begin{enumerate}[{\rm (i)}]
\item $\text{\rm Mwc}(E)$ is a closed subalgebra of $\text{\rm W}(E)$ and is unital 
iff $I_E$ is \text{\rm Mwc}.
\item $\text{\rm aMwc}(E)$ is a closed left ideal in $\text{\rm L}(E)$ 
$($and hence a subalgebra of $\text{\rm L}(E)$$)$, and is unital iff 
$I_E$ is \text{\rm aMwc}.
\end{enumerate}
\end{proposition}

\begin{proof}
Both $\text{\rm Mwc}(E)$ and $\text{\rm aMwc}(E)$ are closed subspaces of $\text{\rm W}(E)$ by 
Proposition~\ref{Lwc-algebra} and \cite[Prop.\,3.6.11\,(i)]{Mey91}. 
The closedness of $\text{\rm Mwc}(E)$ under composition follows from Lemma~\ref{Composition of Mwc}\,(i)--(ii).
The fact that $\text{\rm aMwc}(E)$ is a right ideal in $\text{\rm L}(E)$ follows from Lemma~\ref{Composition of Mwc}\,(ii).
Conditions on $I_E$ making algebras $\text{\rm Mwc}(E)$ and $\text{\rm aMwc}(E)$ to be unital, are trivial.
\end{proof}

\begin{proposition}\label{about modulus T a T'}%no
Let $E^\ast$ be a $\text{\rm KB}$-space, let $F$ be Dedekind complete, 
and $F^\ast \in\text{\rm (SP)}$. If $T\in \text{\rm L}_{ob}(E,F)$ then 
$|T|\in \text{\rm aMwc}(E,F)$.
\end{proposition}

\begin{proof}
Let $(x_n)$ be a disjoint bounded sequence in $E$ and $f_n\stackrel{\text{\rm w}}{\to}f$ in $F^\ast $. 
As  $F^\ast\in\text{\rm (SP)}$, we have $\|f_n-f\|\to 0$. Hence
\begin{equation}\label{(**)}
  |(f_n-f)(|T|x_n)|\le\|f_n -f\|\cdot\||T|x_n\| \to 0.
\end{equation}
By %Assertion~\ref{E' is o-cont}, 
\cite[2.4.14]{Mey91}, $x_n\stackrel{\text{\rm w}}{\to}0$, and hence
$f(|T|x_n)=|T|^\ast (f)(x_n) \to 0$. It follows from \eqref{(**)} that $f_n(|T|x_n)\to 0$, 
and hence $|T|\in\text{\rm aMwc}(E,F)$. 
\end{proof}

\subsection{}
In certain cases $\text{\rm aL(M)wc}$ operators coincide with bounded 
operators.

\begin{proposition}\label{prop 13}%OK
For any Banach space $X$ the following hold.
\begin{enumerate}[\em (i)]
\item
If $F$ is an AL-space then $\text{\rm aLwc}(X,F)=\text{\rm L}(X,F)$.
\item
If $E$ is an AM-space then $\text{\rm aMwc}(E,Y)=\text{\rm L}(E,Y)$.
\end{enumerate}
\end{proposition}

\begin{proof}
(i) Let $S\in\text{\rm L}(X,F)$ and let $C\subseteq X$ be relatively weakly compact. 
Then $S(C)\subseteq F$ is relatively weakly compact. By \cite[Thm.5.56]{AlBu}, 
$S(C)$ is an \text{\rm Lwc} subset of $F$,
and hence $S\in\text{\rm aLwc}(X,F)$.

\medskip
(ii) Let $T\in\text{\rm L}(E,Y)$. Since $E^\ast $ is an AL-space, it follows from (i) that 
$T^\ast \in\text{\rm aLwc}(Y^\ast ,E^\ast )$. Then $T\in\text{\rm aMwc}(E,Y)$ by \cite[Thm.2.5]{BLM18}.
\end{proof}

\subsection{}
The next two theorems describes some conditions for
$\text{\rm aMwc}_+(E,F)\subseteq\text{\rm Mwc}(E,F)$ and $\text{\rm aLwc}_+(E,F)\subseteq\text{\rm Lwc}(E,F)$. 

\begin{definition}\label{E has PSP}
{\em A Banach lattice $E$ has
the {\em positive Schur property} (shortly, $E\in(\text{\rm PSP})$)
if positive \text{\rm w}-null sequences in $E$ are norm null (cf. \cite{Wnuk2013}).
}
\end{definition}

\begin{theorem}\label{necessary conditions for aMwc to be Mwc}%OK
Let $E^\ast$ be a $\text{\rm KB}$-space, let $F$ be Dedekind complete,
and let $\text{\rm aMwc}_+(E,F)\subseteq\text{\rm Mwc}(E,F)$. Then either 
$E\in\text{\rm (PSP)}$ or $F=F^a$.
\end{theorem}

\begin{proof}
Suppose $E\not\in\text{\rm (PSP)}$ and $F\ne F^a$. 
We construct a positive \text{\rm aMwc} operator 
from $E$ to $F$ that is not \text{\rm Mwc}. 
By assumption, there exists a disjoint w-null sequence $(x_n)$ in $E_+$ with $\|x_n\|\not\to 0$. 
WLOG, $\|x_n\|=1$ for all $n$. There exists a sequence $(f_n)$ in $E^\ast_+$ with $\|f_n\|=1$,  
such that 
\begin{equation}\label{1x}
   |f_n(x_n)|\ge\frac{1}{2} \ \ \ \ (\forall n\in{\mathbb N}).
\end{equation}
Define a positive operator $T: E\to\ell^\infty$ by 
\begin{equation}\label{2x}
   Tx:=[(f_n(x))_n] \ \ \ \ (\forall x\in E).
\end{equation}
Since $F$ is Dedekind complete 
and $F\ne F^a$, there exists a lattice embedding 
$S:\ell^\infty\to F$ (cf. \cite[Thm.4.51]{AlBu}). Then, for some $M>0$:
\begin{equation}\label{3x}
   \|Sa\|\ge M\|a\|_\infty \ \ \ (\forall a\in\ell^\infty).
\end{equation}
Since $\ell^\infty\in(\text{\rm DPP})$ (cf. \cite[Thm.5.85]{AlBu}) then $S\in\text{\rm wDP}(\ell^\infty,F)$ 
\cite[p.349]{AlBu}, that is,  for each $g_n\stackrel{\text{\rm w}}{\to}g$ in $F^\ast$
and each $y_n\stackrel{\text{\rm w}}{\to}0$ in $\ell^\infty$:
$$
   g_n(Sy_n)\to 0.
$$ 
Let $(z_n)$ be disjoint bounded in $E$ and let $g_n\stackrel{\text{\rm w}}{\to}g$ in $F^\ast$.
Since $E^\ast$ has o-continuous norm, $z_n\stackrel{\text{\rm w}}{\to}0$ 
in $E$ by~\cite[2.4.14]{Mey91}. %Assertion~\ref{E' is o-cont}. 
Then $Tz_n\stackrel{\text{\rm w}}{\to}0$ in $\ell^\infty$.
It follows from \eqref{3x} that $g_n(STz_n)\to 0$. Therefore, $ST\in\text{\rm aMwc}(E,F)$. 
By utilizing \eqref{1x}, \eqref{2x}, and \eqref{3x}, we obtain
\begin{equation}\label{5x}
   \|STx_n\|=\|S[(f_k(x_n))_k]\|\ge 
   M\|(f_k(x_n))_k\|_\infty\ge M|f_n(x_n)|\ge\frac{M}{2}>0.
\end{equation}
It follows from \eqref{5x} that $ST\not\in\text{\rm Mwc}(E,F)$, and we are done.
\end{proof}

\noindent
The following result is an \text{\rm Lwc} version of 
Theorem~\ref{necessary conditions for aMwc to be Mwc}.

\begin{theorem}\label{sufficient conditions for aLwc subseteq Lwc}%no
If $\text{\rm aLwc}_+(E,F)\subseteq\text{\rm Lwc}(E,F)$, then either 
%$E^\ast$ has \text{\rm o}-co\-n\-ti\-nuous norm 
$E^\ast$ is a $\text{\rm KB}$-space, or $F^\ast \in\text{\rm (PSP)}$.
\end{theorem}

\begin{proof}
Suppose $E^\ast $  is not a $\text{\rm KB}$-space
and $F^\ast \not\in\text{\rm (PSP)}$.
Since the norm in $E^\ast$ is not $\text{\rm o}$-continuous,
there exist a disjoint sequence $(x^\ast_n)$ in $E^\ast $ and $x^\ast \in E^\ast $ such that
$0\le x^\ast_n \le x^\ast $ and $\|x^\ast_n\|=1$ for all $n\in{\mathbb N}$. 
As $F^\ast \not\in\text{\rm (PSP)}$, there is a disjoint w-null sequence $(f_n)$ in $F^\ast_+$ 
with  $\|f_n\|\not\to 0$ by \cite[Thm.7]{Wnuk93}. 
By passing to subsequence and scaling, we may assume that $\|f_n\|\ge 1$ for every $n$. 
Choose $y_n\in F_+$ such that $\|y_n\|=1$ and  
$|f_n(y_n)|\ge\frac{1}{2}$ for all $n\in{\mathbb N}$. 
Define positive operators $T:E\to \ell^1$ and $S:\ell^1\to F$ by 
$$
   Tx:=(x^\ast_k(x))_k \ \ \& \ \ \ Sa:=\sum_{k=1}^\infty a_k y_k .
$$ 
Take a \text{\rm w}-null $(z_n)$ in $E$, and let $(f_n)$ 
be disjoint in $B_{F^\ast}$. As $\ell^1\in(\text{\rm SP})$, then 
$\|Tz_n\|\to 0$, and hence
\begin{equation}\label{6x}
   |f_n(STz_n)|\le\|f_n\|\cdot\|S\|\cdot\|Tz_n\|\to 0.
\end{equation} 
By \cite[Prop.1]{EAS}, \eqref{6x} implies $ST\in\text{\rm aLwc}_+(E,F)$. 

If the operator $ST$ were to be \text{\rm Lwc}, then $(ST)^\ast$ 
would be \text{\rm Mwc}. However,
$$
   \|(ST)^\ast (f_n)\|=\|T^\ast S^\ast (f_n)\|=\|\sum_{k=1}^\infty f_n(y_k)\|\ge |f_{n}(y_n)|\ge\frac{1}{2}
$$
for each $n$, which shows $(ST)^\ast \notin\text{\rm Mwc}(F^\ast ,E^\ast )$ and hence $ST\notin\text{\rm Lwc}(E,F)$.
This contradiction completes the proof.
\end{proof}

\subsection{}
In general, $\text{\rm K}(E,F)\not\subseteq \text{\rm aLwc}(E,F)\cup\text{\rm aMwc}_+(E,F)$, 
as the next example shows.

\begin{example}{\em
Consider $T:\ell^1\to\ell^\infty$ defined as
$$
  T(\alpha)=\left(\sum_{n=1}^\infty\alpha_n\right)e, \quad \text{\rm where} \quad e=(1,1,\dots).
$$
$T$ is compact, but it is neither an aLwc nor aMwc operator. 
}
\end{example}

\noindent
We know from \cite[Thm.\,1]{EAS} that every  Dunford--Pettis (resp. compact)  
$T:X\to F$ is an aLwc operator iff $F=F^a$. 
%If $E^\ast = (E^\ast)^a$, then every compact $T:E\to Y$ is an aMwc operator (see \cite[Thm.\,2]{EAS}). 
The following extends this to almost limited operators.

\begin{proposition}\label{aLim=aLwc}
$F=F^a$ iff $\text{\rm aLim}(X,F)=\text{\rm aLwc}(X,F)$. 
\end{proposition}

\begin{proof}
If almost limited operators are aLwc, then so are the compact operators, and the result follows
from \cite[Thm.\,1]{EAS}.

For the converse, let $(x_n)$ be a w-null sequence in $X$, and let $(f_n)$ be disjoint in $B_{F^\ast}$. 
We need to show that $f_n(Tx_n)\to 0$. 
Since $F=F^a$ then $f_n\stackrel{\text{\rm w}^\ast}{\to}0$ and so $\|T^\ast f_n\|\to 0$ as  $T\in\text{\rm aLim}(X,F)$. 
That $T\in\text{\rm aLwc}(X,F)$ follows from $|T^\ast f_n(x_n)| \leq\|T^\ast f_n\|\cdot\|x_n\|$,
as $(x_n)$ is bounded.
\end{proof}

\begin{proposition}\label{propmn1}%no
The following holds.
\begin{enumerate}[{\em (i)}]
\item
If $F$ has weakly sequentially continuous lattice operations and $F=F^a$, then 
$\text{\rm aLim}(E,F)\subseteq \text{\rm DP}(E,F)$ for every $E$.
\item
If $E$ has weakly sequentially continuous lattice operations, then\\
$\text{\rm aLim}_+(E,F)\subseteq \text{\rm DP}(E,F)$ for every $F$.
\end{enumerate}
\end{proposition}

\begin{proof}
(i) \ 
Let $T\in\text{\rm aLim}(E,F)$. 
Take $x_n \stackrel{\text{\rm w}}{\to}0$ in $E$. We need to check $\|Tx_n\|\to 0$.
By~\cite[Cor.2.6]{DoFr}, % Assertion~\ref{Dodds-Fremlin}\,(i), 
we need to show $|Tx_n| \stackrel{\text{\rm w}}{\to}0$ and $f_n(Tx_n)\to 0$ for each
disjoint bounded $(f_n)$ in $F^\ast_+$.
Since $Tx_n \stackrel{\text{\rm w}}{\to}0$, we have  $|Tx_n| \stackrel{\text{\rm w}}{\to}0$. 
Let $(f_n)$ be disjoint bounded in $F^\ast_+$. 
Since $T\in\text{\rm aLim}(E,F)$ then $ \|T^\ast f_n\|\to 0$.
Since $x_n \stackrel{\text{\rm w}}{\to}0$, for some $M$, $\|x_n\|\leq M$ for all $n$.
It follows from
$$
   |f_n(Tx_n)|\leq  \|T^\ast f_n\| \cdot \|x_n\| \leq M\|T^\ast f_n\|\to 0
$$ 
that $f_n(Tx_n)\to 0$, as desired. 
\medskip

(ii) \ 
Let $T\in\text{\rm aLim}_+(E,F)$ and $x_n \stackrel{\text{\rm w}}{\to}0$ in $E$. 
Then $\|x_n\|\leq M$ for some $M$ and all $n$.
Since $E$  has weakly sequentially continuous lattice operations, $|x_n| \stackrel{\text{\rm w}}{\to}0$. 
It follows from $|f(|Tx_n|)|\leq  |f|(T|x_n|)=(T^\ast |f|)|x_n|\to 0$ for every $f\in F^\ast$
that $|Tx_n| \stackrel{\text{\rm w}}{\to}0$. 
Let $f_n \stackrel{{\rm w}^\ast}{\to}0$ be disjoint bounded in $F^\ast_+$. 
Since $T\in\text{\rm aLim}(E,F)$, we have $ \|T^\ast f_n\|\to 0$.
Since $T\geq 0$, 
$$
 |f_n(Tx_n)|\leq  f_n(T|x_n|)=(T^\ast f_n)|x_n| \leq \|T^\ast f_n\|\|x_n\| 
 \leq M\|T^\ast f_n\|\to 0
$$ 
implies $f_n(Tx_n)\to 0$. Combining this with $|Tx_n| \stackrel{\text{\rm w}}{\to}0$,
we obtain $\|Tx_n\|\to 0$ by~\cite[Cor.2.6]{DoFr}. % Assertion~\ref{Dodds-Fremlin}\,(i). 
Thus, $T\in \text{\rm DP}(E,F)$.
\end{proof}

\subsection{}
The following gives necessary conditions for positive wDP operators to be limited.   

\begin{proposition}\label{propmn4}%no
If $\text{\rm wDP}_+(E,F)\subseteq \text{\rm Lim}(E,F)$
then one of the following holds.
\begin{enumerate}[{\em (i)}]
\item
$E^\ast$ is a \text{\rm KB}-space.
\item
$F^\ast$ has the Schur property.
\end{enumerate}
\end{proposition}

\begin{proof}
Assume neither $E^\ast$  is a \text{\rm KB}-space, nor $F^\ast\in \text{\rm (SP)}$.
Then, by
\cite[Thm.\,4.69]{AlBu}, %Wickstead's proof of Theorem~1 in \cite{Wick96}, 
$E$ contains a sublattice isomorphic to $\ell^1$, 
and there is a positive projection $P$ from $E$ onto $\ell^1$ (see \cite[2.3.11]{Mey91}).
As $F^\ast\not\in\text{\rm (SP)}$, there exists a w-null (hence w$^\ast$-null) sequence
$(f_n)$ in $F^\ast$ such that $\|f_n\|=1$ for all $n$. 
For each $n$, take $y_n\in(B_F)_+$ with $|f_n(y_n)|\geq\frac{1}{2}$.
Let $S:\ell^1 \to F$ be defined as $S(\alpha_n)=\sum_{n=1}^\infty \alpha_n y_n$
and $T=S\circ P$. 

As $\ell^1 \in \text{\rm (DPP)}$,  $T\in\text{\rm wDP}_+(E,F)$ (cf., \cite[Exercise~12 p.\,356]{AlBu}).

Suppose $T\in \text{\rm Lim}(E,F)$. 
Since $P$ is onto, % by Open Mapping theorem, 
there exists some $\delta >0$ with $\delta B_{\ell^1} \subseteq P(B_E)$. Thus
$$
\|T^\ast f_n\| = \sup\limits_{x\in B_E}|P^\ast S^\ast f_n(x)| 
= \sup\limits_{x\in B_E}|f_n(SPx)|\geq
$$
$$
  \sup\limits_{z\in \delta B_{\ell^1}}|f_n( Sz)|\geq  \delta|f_n (S e_n)|
   \geq  \delta|f_n(y_n)|\geq  \frac{\delta}{2}  \quad\quad (\forall n \in{\mathbb N}),
$$
where $(e)_n$ are the basis vectors of $\ell^1$.
Since $f_n\stackrel{\text{\rm w}^\ast}{\to}0$, the operator $T$ is not limited. A contradiction.
\end{proof}

%%%%%%%%%%%%%%%%%%%%
\section{Order \text{\rm L(M)wc} operators}
%%%%%%%%%%%%%%%%%%%%

Order \text{\rm L(M)wc} operators were introduced recently in \cite{BLM20,BLM21}.
%The idea of definition of an order \text{\rm Lwc} operator $T:E\to F$ 
%lies in replacing bounded sets in the domain of \text{\rm Lwc}  operator by order intervals.

\subsection{}

A rank one bounded operator need not to be \text{\rm oLwc}. 
Consider $T:\ell^2\to\ell^\infty$ defined by
$Ta=(\sum\limits_{k=1}^\infty a_k){\mathbb 1}_{\mathbb N}$.
Then $Te_n={\mathbb 1}_{\mathbb N}$ and, as ${\mathbb 1}_{\mathbb N}\not\in(\ell^\infty)^a$, 
then $T\not\in \text{\rm oLwc}(\ell^2,\ell^\infty)$. 

\medskip
\noindent
The following result is an improvement of \cite[Thm.2.5]{BLM21}:

\begin{proposition}\label{regular operators are oMwc}
The following conditions are equivalent.
\begin{enumerate}[\em (i)]
\item 
$E^\ast $ is a $\text{\rm KB}$-space. 
\item 
For every $F$, ${\rm L}_{ob}(E,F)\subseteq\text{\rm oMwc}(E,F)$.
\item 
For every $F$, $\text{\rm L}_r(E,F)\subseteq\text{\rm oMwc}(E,F)$.
\item 
For every $F$, $\text{\rm K}(E,F)\subseteq\text{\rm oMwc}(E,F)$.
\item 
Each rank one bounded operator $T:E\to E$ is \text{\rm oMwc}.
\end{enumerate}
\end{proposition}

\begin{proof} 
(i) $\Longleftrightarrow$ (ii) is in \cite[Thm.2.5]{BLM21} 
and the implication (ii) $\Longrightarrow$ (iii) is trivial. 

\medskip
(iii) $\Longrightarrow$ (i). If not, there is a disjoint bounded $(x_n)$ in $E$ 
and $f\in E^\ast $ such that $f(x_n)\ge 1$ for all $n\in{\mathbb N}$
by~\cite[2.4.14]{Mey91}. % Assertion~\ref{E' is o-cont}. 
The identity $T=I: E\to E$ is \text{\rm oMwc} by $(iii)$.
However, for $(f_n)$ in $F^\ast $ with 
$f_n\equiv f$, $f_n(Tx_n)=f(x_n)\ge 1$, and $f_n(Tx_n)\not\to 0$, 
violating %Assertion~\ref{Bouras--Lhaimer--Moussa - order2}. 
\cite[Thm.2.2]{BLM21}. 
The obtained contradiction proves the implication.

\medskip
(i) $\Longrightarrow$ (iv).
Let $T\in\text{\rm K}(E,F)$ and let $x\in F^\ast_+$. Since $T^\ast \in\text{\rm K}(F^\ast ,E^\ast )$, 
$T^\ast [0,x]$ is relatively compact and almost order bounded. As the norm in $E^\ast $
is \text{\rm o}-continuous, $T^\ast [0,x]$ is an \text{\rm Lwc} subset of $E^\ast $ 
by~\cite[Prop.3.6.2]{Mey91}.
This yields $T^\ast \in\text{\rm oLwc}(F^\ast ,E^\ast )$, and so $T\in\text{\rm oMwc}(E,F)$,
by % Assertion \ref{Bouras--Lhaimer--Moussa - order}\,(i).
\cite[Thm.2.3\,(i)]{BLM21}.

\medskip
(iv) $\Longrightarrow$ (v). It is obvious.

\medskip
(v) $\Longrightarrow$ (i).
If the norm in $E^\ast $ is not \text{\rm o}-continuous, 
take $g\in E^\ast_+\setminus(E^\ast )^a$. 
Let $T$ be an operator defined by $T(x)=g(x)y_0$ with $y_0\in E_+\setminus\{0\}$.
Then $T\in\text{\rm oMwc}(E)$. Choose $f\in E^\ast $ with $f(y_0)=1$.
Since $T^\ast f=f(y_0)g=g\not\in(E^\ast)^a$, the set $T^\ast [0,f]\subseteq E^\ast$ is not \text{\rm Lwc}  
as each \text{\rm Lwc} subset of $E^\ast $ must lie in $(E^\ast)^a$. 
Therefore $T^\ast \not\in\text{\rm oLwc}(E^\ast)$. This yields $T\not\in\text{\rm oMwc}(E)$ 
by \cite[Thm.2.3\,(i)]{BLM21}. %Assertion \ref{Bouras--Lhaimer--Moussa - order}\,(i). 
A contradiction.
\end{proof}

\begin{corollary}\label{rank one op1}
Let  $\dim(E)\ge 1$ and $\dim(F)\ge 1$. Then$:$ 
\begin{enumerate}[\em (i)]
\item $E^\ast$ is a $\text{\rm KB}$-space $\Longleftrightarrow$ 
every rank one operator $T:E\to F$ is \text{\rm oMwc}$;$
\item $F=F^a$ $\Longleftrightarrow$ 
every rank one operator $T:E\to F$ is \text{\rm oLwc}.
\end{enumerate}
\end{corollary}

\begin{proof}
(i) It is a part of Proposition \ref{regular operators are oMwc}.

\medskip
(ii) The implication ($\Longrightarrow$) is obvious.

\medskip
(ii) ($\Longleftarrow$) If not, there exist $y\in F_+$ and a disjoint sequence $(y_n)$ in $[0,y]$,
with $\|y_n\|\not\to 0$. Choose $x\in E_+$ and $f\in E^\ast_+$ such that $f(x)=1$, and define 
$T$ as $Tz=f(z)y$. 
Then $Tx=y$, $(y_n)$ is disjoint in $\text{sol}\,(T[0,x])$, and $\|y_n\|\not\to 0$, violating that $T$ 
must be \text{\rm oLwc}. A contradiction.
\end{proof}

\subsection{}
The following result shows that restriction a regular operator on a principal ideal is \text{\rm oMwc}.

\begin{proposition}\label{T is oMwc}
Let $T\in\text{\rm L}_r(E,F)$ and  $x\in E_+$. 
Then the restriction $T:(I_x,\|.\|_x)\to F$ is \text{\rm oMwc}. 
\end{proposition}

\begin{proof}
We may suppose $T\ge 0$. 
The principal order ideal $I_x$ of $E$ is a Banach lattice equipped 
with the norm $\|y\|_x=\inf\{\lambda: |y|\le\lambda x\}$.
Let $(x_n)$ be a disjoint order bounded sequence in $I_x$. 
Then $x_n \to 0$ in $\sigma(I_x,I^\ast_x)$. 
Since the lattice operations are sequentially 
weakly continuous, %(cf. \cite[Prop.13]{Wnuk2013}), 
we have $|x_n| \to 0$ in $\sigma(I_x,I^\ast_x)$. 
Since $T$ is positive, $T: (I_x,\|.\|_x)\to F$ is norm continuous, and hence 
$T$ is also continuous 
when the both spaces are equipped with their weak topologies. 
Let  $(f_n)$ be order bounded in $F^\ast $, 
say $|f_n|\le f\in F^\ast $. Then 
$$
   |f_n(Tx_n)|\le |f_n|(|Tx_n|)\le |f_n|(T|x_n|)\le f(T|x_n|).
$$
Since $f_n(Tx_n)$ can be made arbitrary small, $T\in \text{\rm oMwc}(I_x,F)$.
\end{proof}

\subsection{}
It follows directly from the definitions:   
$$
   \text{\rm Lwc}(E,F)\subseteq\text{\rm oLwc}(E,F)\cap\text{\rm aLwc}(E,F);
$$
$$
   \text{\rm Mwc}(E,F)\subseteq\text{oMwc}(E,F)\cap\text{aMwc}(E,F).
$$
$$
   I_{c_0}\in\text{\rm oLwc}(c_0)\setminus\text{\rm aLwc}(c_0)\ \text{\rm and}\ 
   I_{\ell^1}\in\text{\rm oMwc}(\ell^1)\setminus\text{\rm aMwc}(\ell^1).
$$
It can be easily shown that 
$I_{\ell^2}\in\text{\rm oMwc}(\ell^2)\setminus\text{\rm aMwc}(\ell^2)$
(cf. Remark~\ref{reflexive not aMW}). 
By Proposition \ref{prop 13}, for each $X$: 
$\text{\rm oLwc}(X,F)=\text{\rm aLwc}(X,F)=\text{\rm L}(X,F)$ 
if $F$ is an AL-space; 
and $\text{\rm oMwc}(E,X)=\text{\rm aMwc}(E,X)=\text{\rm L}(E,X)$ 
if $E$ is an AM-space.

\begin{remark}\label{reflexive not aMW}
{\em It is easily seen that the identity operator $I_E$
\begin{enumerate}[{\rm (i)}]
\item 
in any infinite dimensional reflexive Banach lattice $E$ is not \text{\rm aMwc};
\item 
is an $l$\text{\rm -Mwc} operator iff $E^\ast$ is a KB-space.
\end{enumerate}}
\end{remark}

\medskip
\noindent
It follows directly that $\text{\rm aLwc}(\ell^2)\subseteq\text{\rm oLwc}(\ell^2)$.
Then $\text{\rm aMwc}(\ell^2)\subseteq\text{\rm oMwc}(\ell^2)$ by duality.
%Since $I_{\ell^2}\in\text{\rm oMwc}(\ell^2)\setminus\text{\rm aMwc}(\ell^2)$ 
%then $\text{\rm aMwc}(\ell^2)\subsetneqq\text{\rm oMwc}(\ell^2)$.

%\subsection{}
%These inclusions can be proved using well known duality results.

\begin{proposition}\label{assert_smsth}
The following inclusions hold.
\begin{enumerate}[{\rm (i)}]
\item
$\text{\rm aLwc}_+(E,F) \subseteq \text{\rm oLwc}_+(E,F) = l\text{\rm -Lwc}_+(E,F)$.
\item
$\text{\rm aMwc}_+(E,F) \subseteq \text{\rm oMwc}_+(E,F) \subseteq l\text{\rm -Mwc}_+(E,F)$.
\end{enumerate}
\end{proposition}

\begin{proof}
(i) \ 
Let $0\leq T\in \text{\rm aLwc}(E,F)$. Then $T(E)\subseteq F^a$.
As $F^a$ is an order ideal in $F$, $[0,Tx]\subset F^a$ for $0\leq x \in E$.
Since $[0,Tx]$ is almost order bounded in $F^a$, it is an $\text{\rm Lwc}$ set by 
\cite[Prop.\,3.6.2]{Mey91}. Hence $T[0,x]$ is an $\text{\rm Lwc}$ set since 
$T[0,x]\subseteq [0,Tx]$. Thus $T\in \text{\rm oLwc}$.

First, $\text{\rm oLwc}(E,F) \subseteq l\text{\rm -Lwc}(E,F)$. 
Let  $T\in \text{\rm oLwc}(E,F)$. Then (by~\cite[Thm.\,2.1]{BLM21})  %definition)
 $|T^\ast f_n|\stackrel{\text{\rm w}^\ast}{\to}0$
for each disjoint $(f_n)\subseteq B_{F^\ast}$. 
%We show $T\in l\text{\rm -Lwc}$. 
Note that
$$
 T^\ast f_n(x) \leq |T^\ast f_n(x)| \leq |T^\ast f_n|(|x|) \quad x\in E. 
$$

Next we show $l\text{\rm -Lwc}_+(E,F) \subseteq \text{\rm oLwc}(E,F)$.
Let $0\leq T$. Then by \cite[Thm~4.2]{OM22}, $T$ maps almost order bounded 
sets in $E$ to $\text{\rm Lwc}$ sets in $F$. 
Let $[-x,x]+\varepsilon B_E$ be an almost order bounded set in $E$. 
Then $T([-x,x]+\varepsilon B_E)$ is an $\text{\rm Lwc}$ set in $F$.
As subsets of $\text{\rm Lwc}$ set are $\text{\rm Lwc}$, it follows that
$T[0,x]$ is an $\text{\rm Lwc}$ set for $x\in E_+$ and 
$T\in \text{\rm oLwc}(E,F)$ so 
$$
\text{\rm aLwc}_+(E,F) \subseteq \text{\rm oLwc}_+(E,F) = l\text{\rm -Lwc}_+(E,F).
$$

%\medskip
(ii) \ 
Let $0\leq T\in \text{\rm aMwc}(E,F)$. 
Then $T^\ast\in \text{\rm aLwc}_+(F^\ast,E^\ast)$ by
\cite[Thm.\,2.5]{EAS}, which implies $T^\ast\in \text{\rm oLwc}_+(F^\ast,E^\ast)$. 
Therefore $T\in \text{\rm oMwc}_+(E,F)$ by \cite[Thm.\,2.3]{BLM21}.
Thus $\text{\rm aMwc}_+(E,F) \subseteq \text{\rm oMwc}_+(E,F)$.

 If $T\in \text{\rm oMwc}_+(E,F)$, then $T^\ast\in \text{\rm oLwc}_+(F^\ast,E^\ast)$ 
by \cite[Thm.\,2.3]{BLM21}. Then $T^\ast\in l\text{\rm -Lwc}$ by the previous remarks. 
If  $T^\ast\in l\text{\rm -Lwc}_+(F^\ast,E^\ast)$, then by \cite[Thm.\,4.13]{OM22},
$T\in l\text{\rm -Mwc}_+(E,F)$, and we have the required.
\end{proof}

\subsection{On composition with \text{\rm oL(M)wc} operators.}

\begin{proposition}\label{Composition of oLMwc}%no
The following holds.
\begin{enumerate}[{\rm (i)}]
\item
If $S\in{\rm L}_{ob}(E,F)$ and $T\in\text{\rm oLwc}(F,G)$ 
then $T  S\in\text{\rm oLwc}(E,G)$.
\item
If $S\in{\rm L}_{ob}(F,G)$ and $T\in\text{\rm oMwc}(E,F)$, 
then $S  T\in\text{\rm oMwc}(E,G)$.
\item
If $S\in \text{\rm semi-K}(E,F)$ and $T\in \text{\rm oLwc}(F,G)$,
then $TS\in \text{\rm Lwc}(E,G)$
\end{enumerate}
\end{proposition}

\begin{proof}
(i)\ 
follows directly from the definition of \text{\rm oLwc} operator.
\medskip

(ii) \  
Let $T\in\text{\rm oMwc}(E,F)$ and $S\in{\rm L}_{ob}(F,G)$. 
Then $T^\ast \in\text{\rm oLwc}(F^\ast ,E^\ast )$ 
by \cite[Thm.2.3\,(i)]{BLM21}, %Assertion~\ref{Bouras--Lhaimer--Moussa - order}\,(i), 
and $T^\ast$ is order bounded by \cite[Thm.\,1.73]{AlBu}. 
It follows from (i) that $(S  T)^\ast =T^\ast   S^\ast \in\text{\rm oLwc}(G^\ast ,E^\ast )$.
Applying \cite[Thm.2.3\,(i)]{BLM21} %Assertion~\ref{Bouras--Lhaimer--Moussa - order}\,(i) 
again, we get $S  T\in\text{\rm oMwc}(E,G)$.
\medskip

(iii) \ 
WLOG suppose $\|S\|\leq 1$.
Let $T\in \text{\rm oLwc}(F,G)$ and $\varepsilon>0$. 
Pick $u_\varepsilon\in F_+$ such that
$S(B_E)\subseteq [-u_\varepsilon,u_\varepsilon]+\varepsilon B_F$. Then
$$
   TS(B_E)\subseteq T[-u_\varepsilon,u_\varepsilon]+\varepsilon T(B_F)
                          \subseteq T[-u_\varepsilon,u_\varepsilon]+\varepsilon B_G .
$$
Since $T\in \text{\rm oLwc}(F,G)$, the set $T[-u_\varepsilon,u_\varepsilon]$ 
is an \text{\rm Lwc} subset of $G$. Since $\varepsilon>0$ is arbitrary,  $TS(B_E)$ 
is an \text{\rm Lwc} subset of  $G$ which yields $TS\in \text{\rm Lwc}(E,G)$.
\end{proof}

%\subsection{On span of positive  \text{\rm oL(M)wc} operators.}

\begin{proposition}\label{prop 14}%no
The following statements hold.  
\begin{enumerate}[\em (i)]
\item
If $F$ is an AL-space then $\text{\rm span}(\text{\rm oLwc}_+(E,F))=\text{\rm L}_r(E,F)$.
\item
If $E$ is an AM-space then $\text{\rm span}(\text{\rm oMwc}_+(E,F))=\text{\rm L}_r(E,F)$.
\end{enumerate}
\end{proposition}

\begin{proof}
(i) Let $S=S_1-S_2$ with both $S_1,S_2:E\to F$ positive, and let $x\in X_+$. 
Then $S_1([-x,x])$ and $S_2([-x,x])$ are both relatively weakly compact in $F$.
By \cite[Thm.5.56]{AlBu}, the sets $S_1([-x,x])$ and $S_2([-x,x])$ are both
\text{\rm Lwc} subsets of $F$. Then $S_1$ and $S_2$ are 
positive \text{\rm oLwc} operators and 
hence $S\in \text{\rm span}(\text{\rm oLwc}_+(E,F))$. 

\medskip
(ii) Let $T\in\text{\rm L}_r(E,F)$. Since $E^\ast $ is an AL-space, 
it follows from~(i) that 
$T^\ast \in\text{\rm span}(\text{\rm oLwc}_+(F^\ast ,E^\ast ))$. 
Then $T\in\text{\rm span}(\text{\rm oMwc}_+(E,F))$ by 
\cite[Thm.2.3\,(i)]{BLM21}. %Assertion~\ref{Bouras--Lhaimer--Moussa - order}\,(i).
\end{proof}

%%%%%%%%%%%%%%%%
%%%%%%%%%%%%%%%%%%
\section{On weak compactness}
%%%%%%%%%%%%%%%%%
%%%%%%%%%%%%%%%%%%%

\subsection{}
Here we include some observations on almost limited operators.

\begin{proposition}\label{newmncP1}
The following holds.
\begin{enumerate}[{\em (i)}]
\item
If $F=F^a$ and $F$ has sequentially weakly  continuous lattice operations, then
$\text{\rm aLim}(E,F)\cap l\text{\rm -Mwc}(E,F) \subseteq \text{\rm Mwc}(E,F)$.
\item
If $E^\ast$ is a \text{\rm KB}-space then $\text{\rm aDP}_+(E,F) \subseteq \text{\rm Mwc}(E,F)$.
\end{enumerate}
\end{proposition}

\begin{proof}
(i) \ 
Let $T\in\text{\rm aLim}(E,F)\cap l\text{\rm -Mwc}(E,F)$.
Take a disjoint $(x_n)$ in $B_E$. We need to show $\|Tx_n\|\to 0$. 
As $T\in l\text{\rm -Mwc}(E,F)$, $Tx_n\stackrel{\text{\rm w}}{\to}0$
in $F$, and by the assumption on $F$, $|Tx_n|\stackrel{\text{\rm w}}{\to}0$.
If $(f_n)$ is a disjoint bounded sequence in $F^\ast_+$, then 
$|f_n(Tx_n)|\leq \|T^\ast f_n\|\cdot\|x_n\|$. As $\|T^\ast f_n\|\to 0$, then 
$T\in\text{\rm Mwc}(E,F)$ by~\cite[Cor.2.6]{DoFr}. % Assertion~\ref{Dodds-Fremlin}(i).

\medskip
(ii) \ 
Let $T\in\text{\rm aDP}_+(E,F)$.
Let $(x_n)$ be a disjoint sequence in $B_E$. We need to show $\|Tx_n\|\to 0$. 
Since $E^\ast$ is a KB-space, $x_n\stackrel{\text{\rm w}}{\to}0$ and hence
 $Tx_n\stackrel{\text{\rm w}}{\to}0$ in $F$. As $T\in\text{\rm aDP}(E,F)$ then $\|T x_n\|\to 0$ in $F$.
\end{proof}

\begin{remark}{\em
In general, neither $l$-Lwc nor $l$-Mwc operators are weakly compact.
For example, $T=I_{\ell^\infty}$ (and hence $T^\ast=I_{(\ell^\infty)^\ast}$) is not weakly compact, 
yet $T^\ast\in l\text{\rm -Lwc}((\ell^\infty)^\ast)$ by \cite[Thm.2]{AEG_duality}, 
and hence $T\in l\text{\rm -Mwc}(\ell^\infty)$ by \cite[Thm.4]{AEG_duality}.
Proposition~\ref{newmncP1} shows that those $l$-Mwc operators that are almost limited are weakly compact
as they are Mwc.
}
\end{remark}

\subsection{}
The next result shows the intimate relations between the operators that are considered. 
We need the following definition.

\begin{definition}\label{E has many of}
{\em A Banach lattice $E$ has
\begin{enumerate}[{(i)}]
%\item
%the {\em positive Schur property} (shortly, $E\in(\text{\rm PSP})$)
%if positive \text{\rm w}-null sequences in $E$ are norm null (cf. \cite{Wnuk2013}).
\item
the {\em dual positive Schur property} (or $E\in\text{\rm (DPSP)}$) if each positive 
\text{\rm w}$^\ast$-null sequence in $E^\ast$ is norm null; % (cf. \cite[Def.\,3.3]{AEW});
\item
the {\em positive Grothendieck property} 
(or $E\in\text{\rm (PGP)}$) if each positive 
\text{\rm w}$^\ast$-null sequence in $E^\ast$ is weakly null (cf. \cite[p.\,760]{Wnuk2013}).
\end{enumerate}
}
\end{definition}

\begin{theorem}\label{1.2.3_new}
Suppose $F=F^a$.
\begin{enumerate}[{\em (i)}]
\item %{1.2.3} without E^\ast is KB
If $F\in\text{\rm (PGP)}$ then
$\text{\rm wDP}_+(E,F)\subseteq\text{\rm aDP}(E,F)$ for every $E$.
\item %{1.2.6}
If $E\in\text{\rm (DPSP)}$ $($or $E^\ast \in\text{\rm (PSP)}$$)$ then
$\text{\rm wDP}_+(E,F)\subseteq \text{\rm aDP}(E,F)$.
\item %{1.2.4}
If $E^\ast$ is a $\text{\rm KB}$-space, then %for every $F$,
$\text{\rm wDP}(E,F) \subseteq\text{\rm aMwc}(E,F)$
and \\
$\text{\rm DP}_+(E,F)\subseteq \text{\rm aLim}_+(E,F) \subseteq\text{\rm aDP}(E,F)$.
\end{enumerate}
\end{theorem}

\begin{proof}
(i) \ 
%The first inclusion is obvious (see, e.g., \cite[the last paragraph on p.349]{AlBu}). To see the second one, 
Let $T\in\text{\rm wDP}_+(E,F)$, and let $(x_n)$ be a disjoint w-null sequence in $E$. 
We need to show $\|Tx_n\|\to 0$. By~\cite[Cor.2.6]{DoFr}, % Assertion~\ref{Dodds-Fremlin}\,(i), 
we must show $|Tx_n|\stackrel{\rm w}{\to}0$ and $f_n(Tx_n)\to 0$ for each disjoint bounded $(f_n)$ in $F^\ast_+$.
Since $x_n\stackrel{\rm w}{\to}0$ and disjoint, then $|x_n|\stackrel{\rm w}{\to}0$ by \cite[Thm.\,4.34]{AlBu}. 
As $T$ is positive, $|Tx_n|\leq T|x_n|$, and hence $|Tx_n|\stackrel{\rm w}{\to}0$.

Let now $(f_n)$ be a positive disjoint sequence  in $B_{F^\ast}$. Since $F=F^a$, $f_n\stackrel{{\rm w}^\ast}{\to}0$
and, as $F\in \text{\rm (PGP)}$, $f_n\stackrel{\rm w}{\to}0$.
Since $T\in\text{\rm wDP}(E,F)$,  $f_n(Tx_n)\to 0$ and, by~\cite[Cor.2.6, Cor.2.7]{DoFr}, % Assertion~\ref{Dodds-Fremlin}, 
we have $\|Tx_n\|\to 0$.

\medskip
(ii) \ 
Let  $0\leq T\in \text{\rm wDP}(E,F)$, and let $(x_n)$ be a w-null disjoint sequence in $E$. 
We must show $\|Tx_n\|\to 0$.
As $(x_n)$ is disjoint w-null, $|x_n|\stackrel{\rm w}{\to}0$ by \cite[Thm.~4.34]{AlBu}.
Then $|Tx_n|\leq T|x_n|$ and thus $|Tx_n|\stackrel{\rm w}{\to}0$.
Let $(f_n)$ be disjoint in $B_{F^\ast}$.  
Since $F=F^a$, $f_n\stackrel{{\rm w}^\ast}{\to}0$ by \cite[Cor.\,2.4.3]{Mey91},
and hence $T^\ast f_n \stackrel{{\rm w}^\ast}{\to}0$ in $E^\ast$.
As $E\in\text{\rm (DPSP)}$, $\|T^\ast f_n\|\to 0$.
It follows from $|f_n(Tx_n)|\leq\|T^\ast f_n\|\cdot \|x_n\|$ that $T\in \text{\rm aDP}(E,F)$.

The case $E^\ast \in\text{\rm (PSP)}$ is similar and is omitted.

\medskip
(iii) \ 
Let $T\in \text{\rm wDP}(E,F)$.
Take a disjoint $(x_n)$ in $B_E$ and w-null $(f_n)$ in $F^\ast$. 
As $x_n\stackrel{\rm w}{\to}0$, we have $f_n(Tx_n)\to 0$ by definition of $\text{\rm wDP}$-operators.
Thus, $T\in \text{\rm aMwc}(E,F)$ by~\cite[Prop.\,2]{EG23aff}.

%\medskip
%(v) \ 
To prove the second inclusions, first suppose  $0\leq T\in\text{\rm DP}(E,F)$. 
Let $f_n\stackrel{{\rm w}^\ast}{\to}0$ be disjoint in $B_{F^\ast}$. 
Then $|f_n|\stackrel{{\rm w}^\ast}{\to}0$. 
As $T$ is positive, $|T^\ast f_n|\leq T^\ast|f_n|$, and hence $|T^\ast f_n|\stackrel{{\rm w}^\ast}{\to}0$.
Let $(x_n)$ be a disjoint bounded sequence in $E_+$. 
As $E^\ast$ is a $\text{\rm KB}$-space, $|x_n|\stackrel{\rm w}{\to}0$, and therefore $\|Tx_n\|\to 0$.
By~\cite[Cor.2.6]{DoFr}, % Assertion~\ref{Dodds-Fremlin}\,(i), 
$T\in\text{\rm aLim}(E,F)$ as $\|T^\ast f_n\|\to 0$.

Now let $0\leq T\in\text{\rm aLim}(E,F)$. To see $T\in\text{\rm aDP}(E,F)$,
take a disjoint w-null sequence $(x_n)$ in $E$ (then $|x_n|\stackrel{\rm w}{\to}0$ by \cite[Thm.~4.34]{AlBu}).
We need to show $\|Tx_n\|\to 0$.
By~\cite[Cor.2.6, Cor.2.7]{DoFr}, % Assertion~\ref{Dodds-Fremlin}, 
it suffices to show that 
$f_n(T x_n)\to 0$ for every disjoint $(f_n)$ in $(B_{F^\ast})_+$.
Let $(f_n)$ be disjoint in $(B_{F^\ast})_+$.
Since $F=F^a$, $f_n\stackrel{{\rm w}^\ast}{\to}0$ by \cite[Cor.\,2.4.3]{Mey91}. 
Since  $|x_n|\stackrel{\rm w}{\to}0$ then  $T|x_n|\stackrel{\rm w}{\to}0$,
and hence $|Tx_n|\stackrel{\rm w}{\to}0$ because $|Tx_n|\leq T|x_n|$.
Since $T\in\text{\rm aLim}(E,F)$, then $\|T^\ast f_n\|\to 0$.
It follows from $|f_n(T x_n)|\leq \|T^\ast f_n\|\cdot \|x_n\|$ that $f_n(T x_n)\to 0$, as desired.
\end{proof}

\medskip

\begin{corollary}\label{cor to 4.2.2}
%Let $E^\ast$ be a $\text{\rm KB}$-space and let $F=F^a$. 
Let $E^\ast$ and $F$ have order continuous norm.
Then the following holds.
\begin{enumerate}[{\em (i)}]
\item 
$\text{\rm aLwc}_+(E,F) \subseteq  \text{\rm Mwc}(E,F)$.
\item
$\text{\rm DP}_+(E,F)\subseteq \text{\rm Mwc}(E,F)$ and 
$\text{\rm aLim}_+(E,F)\subseteq \text{\rm Mwc}(E,F)$. 
\item
If $F$ has \text{\rm w}-sequentially continuous lattice operations, then \\
$\text{\rm aLwc}(E,F)\subseteq \text{\rm Mwc}(E,F)$.
\item
If $E\in\text{\rm (SP)}$, then
$\text{\rm oLwc}(E,F)\subseteq \text{\rm semi-K}(E,F)$.
\end{enumerate}
\end{corollary}

\begin{proof}
(i) \ 
By \cite[Prop.~2.3]{BLM18}, $\text{\rm aLwc}_+(E,F)\subseteq\text{\rm aDP}(E,F)$.
On the other hand, as $E^\ast$ is a $\text{\rm KB}$-space, 
$\text{\rm aDP}_+(E,F)\subseteq \text{\rm Mwc}(E,F)$ 
by \cite[Prop.~2.1]{AqEl11}.

\medskip
(ii) \ 
By \cite[Prop.~2.2]{AqEl11}, order bounded $\text{\rm aDP}(E,F)$ are $\text{\rm Mwc}$.
If, in addition, $F=F^a$, then
$\text{\rm DP}_+(E,F)\subseteq \text{\rm aLim}_+(E,F)\subseteq \text{\rm aDP}_+(E,F)$
by Proposition~\ref{1.2.3_new}\,(iv).

\medskip
(iii) \ 
Let $T\in\text{\rm aLwc}(E,F)$ and let $(x_n)$ be a disjoint sequence in $B_E$. We need to show $\|Tx_n\|\to 0$. 
As $E^\ast$ is a $\text{\rm KB}$-space, $x_n\stackrel{\rm w}{\to}0$ and $Tx_n\stackrel{\rm w}{\to}0$.
Since $F$ has weak sequentially continuous lattice operations,
$|Tx_n|\stackrel{\rm w}{\to}0$.
By \cite[Cor.\,2.6]{DoFr}, it remains to show $f_n(Tx_n)\to 0$
for each disjoint bounded $(f_n)$ in $F^\ast_+$. 
This holds by the definition of aLwc operators.

\medskip
(iv) \ 
Let $T\in\text{\rm oLwc}_+(E,F)$. 
Let $(f_n)$ be a disjoint w$^\ast$-null sequence in $B_{F^\ast}$.
Then $|T^\ast f_n|\stackrel{w^\ast}{\to}0$. If $(x_n)$ is a disjoint bounded
sequence in $E_+$, $x_n\stackrel{w^\ast}{\to}0$ as $E^\ast$ is a $\text{\rm KB}$-space
and $f_n(Tx_n)=T^\ast f_n(x_n)\to 0$, and $|f_n(Tx_n)|\leq \|f_n\|\cdot \|Tx_n\|\to 0$.
\end{proof}

\subsection{}

We finish with weak compactness of limitedly Lwc operators. 
Let $T\in l\text{\rm -Lwc}_+(E,F)$.
Since $T(E)\subseteq F^a$, we may assume $T$ takes values
in a Banach lattice with order continuous norm (see proof of Proposition~\ref{assert_smsth}\,(i)). This being the case, if $T\in\text{\rm aLim}(E,F)$
then $T\in\text{\rm Lwc}_+(E,F)$ and is weakly compact.

\begin{proposition}
Let $E^\ast$ be a $\text{\rm KB}$-space, $E\in \text{\rm (w$^\ast$DPP)}$,
and $F\in\text{\rm (d)}$.
Then $l\text{\rm -Lwc}_+(E,F) \subseteq \text{\rm W}(E,F)$.
\end{proposition}

\begin{proof}
Let $0\leq T\in l\text{\rm -Lwc}(E,F)$.
First, we show $T\in\text{\rm aLim}(E,F)$. 
By \cite[Thm.\,3]{El}, it suffices to show $\|T^\ast f_n\| \to 0$
for each positive disjoint w$^\ast$-null sequence in $B_{F^\ast}$. 
Let $(f_n)$ be such a sequence.
Then $|T^\ast f_n| = T^\ast f_n\stackrel{{\rm w}^\ast}{\to} 0$. 
By \cite[Cor.\,2.6]{DoFr}, it remains to show $f_n(Tx_n)\to 0$
for each disjoint bounded sequence $(x_n)$ in $E_+$. 
Let $(x_n)$ be such a sequence.
Since $E^\ast$ is a $\text{\rm KB}$-space, $x_n\stackrel{\rm w}{\to}0$. Then $T^\ast f_n(x_n)=f_n(Tx_n)\to 0$
as a Banach lattice with  w$^\ast$DPP are exactly those
for which $f_n(x_n)\to 0$ for each w$^\ast$-null $(f_n)$
in $E^\ast$ and w-null $(x_n)$ in $E$.
Thus  $T\in\text{\rm aLim}(E,F)$. 
Since $T\in l\text{\rm -Lwc}(E,F)$, $T(E)\subseteq F^a$ by \cite[Thm.\,2]{AEG_duality}.
Then $T\in\text{\rm Lwc}(E,F)$ (see \cite[Thm.\,2.6]{CCJ}), and hence $T$ is weakly compact.
\end{proof}

%%%%%%%%%%%%%%%%%%%%%%%%%%%%%%%%%%%%%%%%%%
%%%%%%%%%%%%%%%%%%%%%%%%%%%%%%%%%%%%%%%%%%

%\newpage
{\small
%%%%%%%%%%%%%%%%%%%%
%\begin{thebibliography}{60}
%%%%%%%%%%%%%%%%%%%%
%\bibliographystyle{amsalpha}
\bibliographystyle{plain}
\bibliography{BIBLIMITED_22.05.2024}

%\end{thebibliography}
%}
\end{document}